\documentclass[10pt,reqno]{article}

\usepackage{longtable,latexsym,epsfig,caption,subcaption,microtype,textcomp,stmaryrd,listings,lineno,hyperref}

\usepackage{amssymb}
\usepackage{graphicx}
\usepackage{amscd}
\usepackage{amsthm}
\usepackage{cite}
\usepackage{float}
\usepackage{graphics, amsmath}
\usepackage{amsfonts}
\usepackage{latexsym}
\usepackage{epsf}
\usepackage{xcolor}
\begin{document}

\theoremstyle{plain}
\newtheorem{theorem}{Theorem}
\newtheorem{corollary}[theorem]{Corollary}
\newtheorem{lemma}{Lemma}
\newtheorem*{lemma*}{Lemma}
\newtheorem{proposition}[theorem]{Proposition}

\theoremstyle{definition}
\newtheorem{definition}[theorem]{Definition}
\newtheorem{example}[theorem]{Example}
\newtheorem{conjecture}[theorem]{Conjecture}

\theoremstyle{remark}
\newtheorem{remark}{Remark}
\begin{center}
{\Large{{\bf Normalized Gompertz wavelets and their application  }}}
\medskip
\end{center}
\leftline{}
\leftline{\bf Grzegorz Rz\c{a}dkowski }

\leftline{Department of Finance and Risk Management, Warsaw University of Technology, Narbutta 85, 02-524 Warsaw, Poland}
\leftline{e-mail: grzegorz.rzadkowski@pw.edu.pl}

\newcommand{\Eulerian}[2]{\genfrac{<}{>}{0pt}{}{#1}{#2}}

\begin{abstract}
In the present paper, we define the Gompertz wavelets and show their basic properties. In particular, we prove that the admissibility condition holds for them. We also compute the normalizing factors in the space of square intergrable functions  $L^{2}(\mathbb{R})$ and present an explicit formula for them in terms of the Bernoulli numbers. Then, after implementing the second-order Gompertz wavelets into Matlab's Wavelet Toolbox, we apply them to study the spread of the Covid-19 pandemic in Saudi Arabia.
\end{abstract}



\section{Introduction} 

The Gompertz function is described by the following autonomous differential equation of the first order
\begin{equation}\label{2a}
x'(t)=sx\log\frac{x_{max}}{x} \quad \quad  x(0)=x_0>0, 
\end{equation}
with parameters $s-$growth rate and $x_{max}-$saturation level (asymptote), $0<x_0<x_{max}$;  $\log$ is the natural logarithm.  After solving (\ref{2a}) we can write the Gompertz function in the following  convenient form
\begin{equation}\label{2b}
x(t)=x_{max}e^{-e^{-s(t-t_0)}}, 
\end{equation}
where constant $t_0$ appears in the integration process of (\ref{2a}) and is connected with the initial condition $x(0)=x_0=x_{max}e^{-e^{st_0}}$, thus 
$t_0=\frac{1}{s}\log\log (x_{max}/x_0)$. It is easy to check that $t_0$ is also the inflection point of $x(t)$ (\ref{2b}).

Function (\ref{2b}) was first described and applied in actuarial mathematics in 1825 by Mr. Benjamin Gompertz \cite{G}.  Since then, the Gompertz function has found applications in probability theory (Gumbel distribution), biology, medicine, economics, engineering, physics and many other fields. The first hundred years of the use of this function are well described by Winsor \cite{Wi}. The interesting story of the next almost one hundred years can be found in the article by Tjørve and Tjørve \cite{TT}.

In recent years, many articles have appeared in which the Gompertz function was used to describe Covid-19 cases (infected people who were tested positive). Ohnishi et al \cite{ONF} showed that the first waves of Covid-19 cases in 11 selected countries (Japan, USA, Russia, Brazil, China, Italy, Indonesia, Spain, South Korea, UK, and Sweden) can be modeled using the Gompertz function. They also compared the mechanism of the appearance of the Gompertz function with the mechanism of the time dependence of the number of pions produced in nucleus-nucleus collisions, which is also described by the Gompertz function. Dhahbi et al \cite{DCB} used the Gompertz model to describe the first wave of cases in Saudi Arabia. Kundu et al \cite{KBS} proposed an
automated COVID-19 detection system based on convolution neural networks using the Gompertz function.  Estrada and Bartesaghi \cite{EB} linked the networked SIS model with the Gompertz function.

The structure of the article is as follows. In Sec.~\ref{sec2} we describe the basic properties of the Gompertz function and its derivatives. Sec.~\ref{sec3} is devoted to the Gompertz wavelets based on the second derivative, that we use later. Then, in the same section, we define the Gompertz wavelets for any derivative and prove that the admissibility condition holds for them. In Sec.~\ref{sec4} we present some applications of the Gompertz wavelets, in particular for modeling the spread of the Covid-19 pandemic. The paper is concluded in Sec.~\ref{sec5}. All the data, which we analyze, were taken from the website Our World in Data \cite{Our}. 

\section{The Gompertz function and its derivatives}\label{sec2}
By  $ { n \brace k}$ we denote the Stirling number of the second kind (for subsets), which is defined as the number of ways of partitioning a set of $n$ 
elements into $k$ nonempty subsets; see Graham et al \cite{GKP} and Sloane \cite{OEIS} (sequence A008277). The sequence has the boundary conditions: ${ n \brace 0} =0 $ if $n>0$, ${ 0 \brace 0} =1 $,  ${ n \brace k} =0 $ for $k>n$ or $k<0$.
Let us recall that the numbers fulfill 
\begin{align*}
{ n \brace k} 
	&=\frac{1}{k!}\sum_{j=0}^{k}(-1)^{k-j}{k \choose j}j^n=
	\frac{1}{k!}\sum_{j=0}^{k}(-1)^{j}{k \choose j}(k-j)^{n},\\
	{ n+1 \brace k} &=k{ n \brace k}+ { n \brace k-1},
\end{align*}
and appear in the Taylor expansion
\begin{equation*}
\frac{(e^{w}-1)^{k}}{k!}=	\sum_{n=k}^{\infty }{ n \brace k} 
\frac{w^{n}}{n!}.
\end{equation*}
First few Stirling numbers for subsets are given in Table~\ref{tab1}.
\begin{table}
\caption{First few Stirling numbers of the second kind (Stirling numbers for subsets)  ${n \brace k }$}\label{tab1}
\begin{center}
\begin{tabular}{c|ccccccccc}
\hline
$n\backslash k$ & 0 & 1 & 2 & 3 & 4 & 5 &6 &7\\
\hline
0  & 1 &  &  & &  &  &&\\ 
1  & 0 & 1 & &  &  &  &&\\ 
2  & 0 & 1 & 1 &  &  &  &&\\ 
3  & 0 & 1 & 3 & 1 &  &  &&\\ 
4 & 0 & 1 & 7 & 6 & 1 & & &\\
5 & 0  & 1 & 15 & 25 & 10 & 1& &\\
6& 0 & 1  & 31 & 90 & 65 & 15 &1&\\
7 & 0  & 1 & 63 &  301&  350& 140 &21&1\\			
\hline
\end{tabular}
\end{center}
\end{table}

Rz\c{a}dkowski et al \cite{Rz1} proved, among other results, that if $x(t)$ is a solution of the equation (\ref{2a}) then its
$n$th derivative $x^{(n)}(t)$ can be expressed as
\begin{equation}\label{2c}
x^{(n)}(t)=s^nx\sum\limits_{k=1}^{n}(-1)^{n-k}{n \brace k}\log^{k}\frac{x_{max}}{x}.
\end{equation}
For example
\begin{align}
x''(t)&=s^2x\log\frac{x_{max}}{x}\left(-1+\log\frac{x_{max}}{x}\right),\label{2d}\\
x'''(t)&=s^3x\log\frac{x_{max}}{x}\left(1-3\log\frac{x_{max}}{x}+\log^2\frac{x_{max}}{x}\right).\label{2e}
\end{align}
From formula (\ref{2d}) we obtain the well-known property of the Gompertz function that its value at the inflection point $t_0$ equals $x(t_0)=x_{max}/e \approx 0.368x_{max}$. Similarly, denoting by $t_1$ the smaller of two zeros of the third derivative (\ref{2e}), we get $x(t_1)=\exp(-(3+\sqrt{5})/2)\approx 0.0729x_{max} $. Comparing this with (\ref{2b}) and solving equation
\[ x(t_1)=x_{max}\exp(-(3+\sqrt{5})/2)=x_{max}\exp(-\exp(-s(t_1-t_0)))\] we obtain (cf. Figure~\ref{fig1})
\begin{equation*}
t_1=t_0-\frac{1}{s}\log\frac{3+\sqrt{5}}{2}.
\end{equation*}
Further comments on this can be found in paper  Rz\c{a}dkowski et al \cite{Rz2}.

\begin{figure}
	\begin{center}
	 \includegraphics[height=6cm, width=8cm]{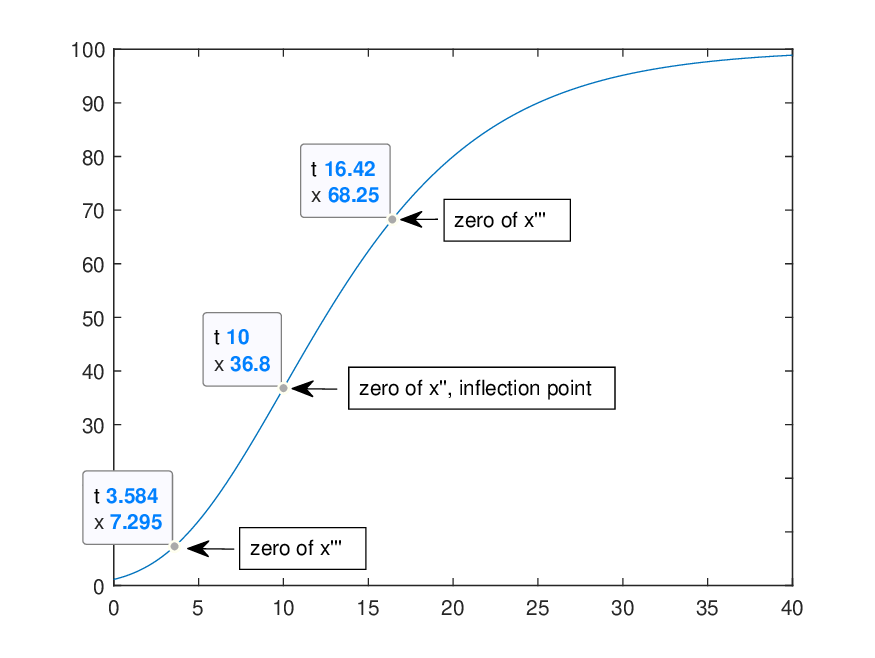}
	\end{center}
	\vspace{0mm}
	\caption{Exemplary Gompertz function with parameters $x_{max}=100, s=0.15, t_0=10$ }
	\label{fig1}
\end{figure}

\section{The Gompertz wavelets}\label{sec3}
\subsection{Wavelets}
We briefly outline the basic general properties of wavelets (cf. \cite{D, MR, M}), which we will need later. A wavelet or mother wavelet (see Daubechies \cite{D}, p.24 ) is a function $\psi \in L^{1}(\mathbb{R})$ such that the following
admissibility condition holds:
\begin{equation}\label{3a}
	C_{\psi}=2\pi\int_{-\infty}^{\infty}|\xi|^{-1}|\widehat{\psi}(\xi)|^2 d\xi < \infty,
\end{equation}
where $\widehat{\psi}(\xi)$ is the Fourier transform $F(\psi)$ of $\psi$, i.e., 
\[
	F(\psi)(\xi)=\widehat{\psi}(\xi)=\frac{1}{\sqrt{2\pi}}\int_{-\infty}^{\infty}\psi(x)e^{-i\xi x} dx. \]
	
Since for $\psi \in L^{1}(\mathbb{R})$, $\widehat{\psi}(\xi)$ is continuous then condition (\ref{3a}) is only satisfied if $\widehat{\psi}(0)=0$, which is equivalent to $\int_{-\infty}^{\infty}\psi(x)dx=0$. On the other hand, Daubechies \cite{D}, p.24  points out that condition $\int_{-\infty}^{\infty}\psi(x)dx=0$  together with a slightly stronger than the integrability condition  $\int_{-\infty}^{\infty}|\psi(x)|(1+|x|)^{\alpha}dx<\infty$, for some $\alpha>0$  are sufficient for (\ref{3a}). Usually, in practice much more is assumed for the function $\psi$ hence, from a practical point of view, conditions $\int_{-\infty}^{\infty}\psi(x)dx=0$ and (\ref{3a}) are equivalent. Suppose the function $\psi$  is also square-integrable, $\psi \in L^{2}(\mathbb{R})$ with the norm
\[ || \psi ||=\left(\int_{-\infty}^{\infty}|\psi(x)|^2 dx \right)^{1/2}. \]

Having a mother wavelet we can generate a doubly-indexed family of wavelets (called children wavelets), by dilating and translating
	\[\psi^{a,b}(x)=\frac{1}{\sqrt{|a|}}\psi\Big(\frac{x-b}{a}\Big),
\]
where $a,b\in \mathbb{R}, \; a\neq 0$. The normalization has been chosen so that $||\psi^{a,b}||=||\psi||$ for all $a, b$. It is assumed usually that $||\psi||=1 $. The continuous wavelet transform (CWT) of a function $f \in L^{2}(\mathbb{R})$ for this wavelet
family is defined as
\begin{equation}\label{3b}
(T^{wav}f)(a,b)=\langle f, \psi^{a,b} \rangle =\int_{-\infty}^{\infty}f(x) \psi^{a,b}(x) dx.
\end{equation}

\subsection{Wavelets based on the second derivative of the Gompertz function}
Consider second derivative of the Gompertz function (\ref{2b}) with parameters $x_{max}=1, s=1, t_0=0$, $x(t)=e^{-e^{-t}}$. Since
 $x'(t)=-x\log x=e^{-e^{-t}}e^{-t} $, then by (\ref{2d}) or directly we get
\begin{equation}\label{3c}
	x''(t)=x\log x(1+\log x)=e^{-e^{-t}}e^{-t} (e^{-t}-1).
\end{equation}

Note that $x'(t)=e^{-e^{-t}}e^{-t} $ is also the probability density function (pdf) of the Gumbel distribution. We calculate the three, following integrals related to (\ref{3c}), each by substituting $x=e^{-e^{-t}}, x'(t)=-x\log x=e^{-e^{-t}}e^{-t}$:
\begin{align}
	&\int_{-\infty}^{\infty}x''(t) dt=\int_{-\infty}^{\infty}e^{-e^{-t}}e^{-t} (e^{-t}-1)dt= -\int_{0}^{1}(\log x+1)dx=(-x\log x)|_{0}^{1}=0,\label{3d}\\
	&\int_{-\infty}^{\infty}|x''(t)| dt=\int_{-\infty}^{\infty}e^{-e^{-t}}e^{-t} |e^{-t}-1| dt=\int_{-\infty}^{0}e^{-e^{-t}}e^{-t} (e^{-t}-1) dt+\int_{0}^{\infty}e^{-e^{-t}}e^{-t} (1-e^{-t}) dt\nonumber\\
	&=-\int_{0}^{1/e}(\log x+1)dx+\int_{1/e}^{1}(\log x+1)dx=\frac{1}{e}+\frac{1}{e}=\frac{2}{e},\label{3e}\\
	&\int_{-\infty}^{\infty}(x''(t))^2 dt=\int_{-\infty}^{\infty}(e^{-e^{-t}})^2(e^{-t})^2 (e^{-t}-1)^2dt=-\int_{0}^{1}x\log x (1+\log x)^2 dx=\frac{1}{8}. \label{3f}
\end{align}

Let us define now the Gompertz mother  wavelet $\psi_2(t)$ (Figure~\ref{fig2}) by:
\begin{equation}\label{3g}
\psi_2(t)=2\sqrt{2}x''(t)=2\sqrt{2}e^{-e^{-t}}e^{-t} (e^{-t}-1), \quad t\in \mathbb{R}.
\end{equation}

\begin{figure}
	\begin{center}
	 \includegraphics[height=6cm, width=8cm]{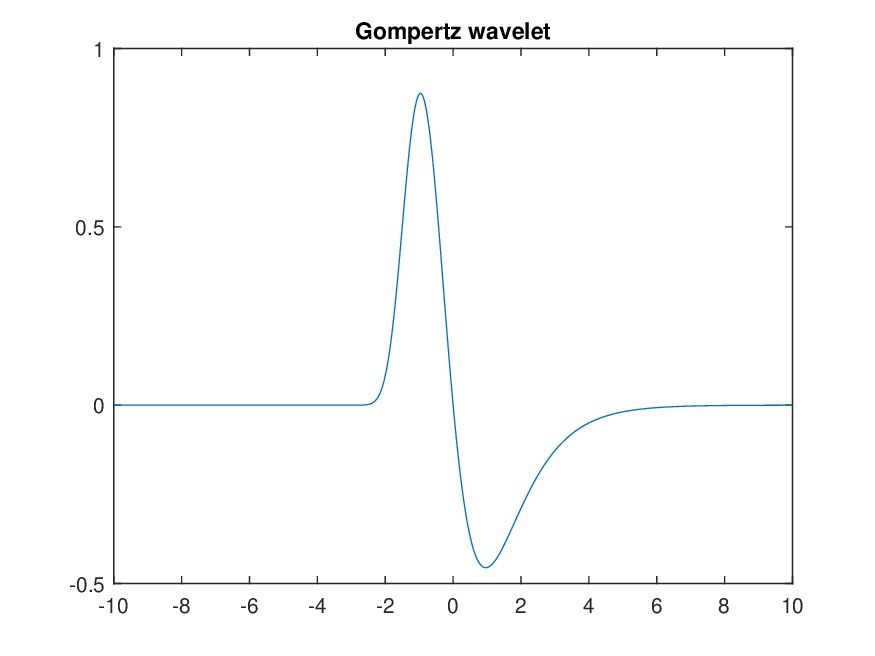}
	\end{center}
	\vspace{0mm}
	\caption{Gompertz mother wavelet $\psi_2(x)$ }
	\label{fig2}
\end{figure}

By definition (\ref{3g}) and (\ref{3d}-\ref{3f}) we have 
\[\int_{-\infty}^{\infty}\psi_2(t) dt=0,\; ||\psi_2||=1,\; \psi_2(t)\in L^{1}(\mathbb{R})\cap L^{2}(\mathbb{R}).\]

Although it is easy to prove that in our case, the sufficient conditions for the admissibility condition (\ref{3a}), described in the previous subsection are fulfilled, we will show a direct and interesting calculation to prove (\ref{3a}). Let us start with an observation concerning the Euler Gamma function
\begin{equation}\label{3h}
\Gamma (z)=\int_{0}^{\infty}x^{z-1}e^{-x} dx,\quad Re(z)>0.
\end{equation}
After substituting in (\ref{3h}) $z=1+i\xi $ and changing the integration variable $x=\exp(-t)$ we obtain
\begin{equation*}
\Gamma (1+i\xi)=\int_{0}^{\infty}x^{i\xi}e^{-x} dx=\int_{\infty}^{-\infty}e^{-i\xi t } e^{-e^{-t}} (-e^{-t}) dt =
\int_{-\infty}^{\infty} e^{-e^{-t}} e^{-t}e^{-i\xi t } dt,
\end{equation*}
which shows that the Fourier transform of the first derivative of the Gompertz function $x'(t)= e^{-e^{-t}}e^{-t}$ is as follows
\begin{equation}\label{3j}
F(x')(\xi)=\widehat{x'}(\xi)=\frac{1}{\sqrt{2\pi}}\Gamma (1+i\xi).
\end{equation}

Using the formula for the Fourier transform of a derivative,  (\ref{3j}) and then definition  (\ref{3g}) we get 
\begin{equation*}
F(x'')(\xi)=\widehat{x''}(\xi)=\frac{i\xi}{\sqrt{2\pi}}\Gamma (1+i\xi),
\end{equation*}
and
\begin{equation}\label{3m}
F(\psi_2)(\xi)=\widehat{\psi_2}(\xi)=\frac{2i\xi}{\sqrt{\pi}}\Gamma (1+i\xi).
\end{equation}

Now we can show that for the Gompertz mother wavelet $\psi_2(t)$ the admissibility condition (\ref{3a}) is satisfied and even the integral can be expressed in a closed form in terms of the Riemann Zeta function $\zeta (z)$. Namely, using (\ref{3m}), the well-known property of the Gamma function
\begin{equation}\label{3k}
 |\Gamma (1+i\xi)|^2=\Gamma (1+i\xi)\Gamma (1-i\xi) =\frac{\pi\xi}{\sinh \pi\xi}, 
 \end{equation}
and the following formula from Dwight's Tables \cite{Dw} (item no $860.502$):
 	\begin{equation*}
		\int_{0}^{\infty}\frac{x^{2}}{\sinh ax} dx=\frac{7}{2a^{3}}\zeta(3),
	\end{equation*}
we obtain
\begin{align}\label{3n}
	C_{\psi_2}&=2\pi\int_{-\infty}^{\infty}|\xi|^{-1}|\widehat{\psi_2}(\xi)|^2 d\xi =
	2\pi\int_{-\infty}^{\infty}\frac{1}{|\xi|}\left|\frac{2i\xi}{\sqrt{\pi}}\Gamma (1+i\xi)\right|^2d\xi=
	2\pi\int_{-\infty}^{\infty}\frac{1}{|\xi|}\frac{4\xi^2}{\pi}\frac{\pi\xi}{\sinh \pi \xi}d\xi \\
	&=2\pi\int_{-\infty}^{\infty}\frac{4\xi^2}{|\sinh \pi\xi |} d\xi=\frac{56\zeta (3)}{\pi^2}<\infty .
\end{align}

We can now generate a doubly-indexed family of the Gompertz wavelets from the mother Gompertz wavelet $\psi_2$ by dilating and translating
	\[\psi_2^{a,b}(t)=\frac{1}{\sqrt{a}}\psi_2\Big(\frac{t-b}{a}\Big),
\]
where $a,b\in \mathbb{R}, \; a> 0$.

\subsection{Wavelets based on higher derivatives of the Gompertz function}
Similarly, as in the previous subsection, we consider the $n$th ($n=3,4,...$) derivative $x^{(n)}(t)$ of the Gompertz function  $x(t)=e^{-e^{-t}}$ with parameters $x_{max}=1, s=1, t_0=0$.  For this particular case formula (\ref{2c}) reads
\begin{equation}\label{4a}
x^{(n)}(t)=(-1)^{n}x(t)\sum\limits_{k=1}^{n}{n \brace k}\log^{k}x(t).
\end{equation}

Since
\begin{align*}
\int_{-\infty}^{\infty}|x^{(n)}(t)| dt=&\int_{-\infty}^{\infty}\Big|(-1)^{n}x(t)\log x(t)\sum\limits_{k=1}^{n}{n \brace k}\log^{k-1}x(t)\Big| dt
=\int_{0}^{1}\Big|\sum\limits_{k=1}^{n}{n \brace k}\log^{k-1}x\Big| dx \\
&\le \sum\limits_{k=1}^{n}{n \brace k}\int_{0}^{1}\Big|\log^{k-1}x\Big| dx=
 \sum\limits_{k=1}^{n}{n \brace k}(k-1)! < \infty 
 \end{align*}
we see that $x^{(n)}(t) \in  L^{1}(\mathbb{R})$.

Rz\c{a}dkowski et al \cite{Rz1} (Theorem 3.2, p. 377) proved the following formula for derivatives of the Gompertz function $x(t)=e^{-e^{-t}}$:
\begin{equation}\label{4b}
\int_{-\infty}^{\infty}(x^{(n)}(t))^2 dt=(-1)^{n}\frac{B_{2n}(1-2^{2n})}{2n}=\frac{|B_{2n}|(2^{2n}-1)}{2n},\quad n=1,2,\ldots,
\end{equation}
where $B_{n}$ is the $n$th Bernoulli number. The Bernoulli numbers are well described in the book by Duren \cite{Dr}. For the convenience of the reader, we sketch here only some of their basic properties.  

The Bernoulli numbers $B_{n}, n=1,2,\ldots$ have the following exponential  generating function 
\begin{equation*}
	B_0+B_{1}z + B_{2}\frac{z^{2}}{2!}+\cdots =\frac{z}{e^{z}-1},
	\qquad |z|<2\pi.
\end{equation*}
and vanish for all odd $n\ge 3$. The numbers are rational and they appear in relations such that
\[\sum_{k=1}^{m-1}k^{n}=\frac{1}{n+1}\sum_{j=0}^{n}{n+1 \choose j} B_{j}m^{n+1-j},\quad m,n\ge 1,
\]
or	
	\[\sum_{k=1}^{\infty}\frac{1}{k^{2n}}=(-1)^{n+1}
	\frac{2^{2n-1}\pi^{2n}}{(2n)!}B_{2n}, \quad n=1,2,\ldots
\]
The first few nonzero Bernoulli numbers are as follows
	\[B_{0}=1,\;B_{1}=-\frac{1}{2},\;B_{2}=\frac{1}{6},\;B_{4}=-\frac{1}{30}
,\;B_{6}=\frac{1}{42},\;B_{8}=-\frac{1}{30},\;B_{10}=\frac{5}{66}.
\]
Note that in case $n=2$ formula (\ref{4b}) agrees with calculation (\ref{3f}) because
\[(-1)^{2}\frac{B_{4}(1-2^{4})}{4}=\frac14 \cdot \Big(-\frac{1}{30} \Big)\cdot (-15)=\frac18. \]

We can define now the Gompertz mother wavelet  $\psi_n(t)$, related to the $n$th derivative of the Gompertz function (\ref{4a}) as
\begin{equation}\label{4c}
    \psi_n(t)=\Big(\frac{2n}{|B_{2n}|(2^{2n}-1)}\Big)^{1/2} x^{(n)}(t).
\end{equation}
Because of (\ref{4b}) the $L^2$ norm $|| \psi_n||=1$. 

We will show now that, for $ \psi_n$ the admissibility condition (\ref{3a}) is fulfilled. Applying in (\ref{3j}) $(n-1)$-times the formula for the Fourier transform of a derivative we get
\begin{equation*}
F(x^{(n)})(\xi)=\widehat{x^{(n)}}(\xi)=\frac{(i\xi)^{n-1}}{\sqrt{2\pi}}\Gamma (1+i\xi),
\end{equation*}
which gives by definition (\ref{4c})
\begin{equation}\label{4d}
F(\psi_{n})(\xi)=\widehat{\psi_{n}}(\xi)=\Big(\frac{n}{|B_{2n}|(2^{2n}-1)\pi}\Big)^{1/2}(i\xi)^{n-1}\Gamma (1+i\xi).
\end{equation}
Using  (\ref{3k}) and the following formula from Dwight's Tables \cite{Dw} (item no $860.509$):
 	\begin{equation*}
		\int_{0}^{\infty}\frac{x^{p-1}}{\sinh ax} dx=\frac{2\Gamma(p)}{a^{p}}\Big( 1-\frac{1}{2^p}\Big)\zeta(p), \quad a>0, p>1
	\end{equation*}
	we obtain
\begin{align}\label{}
	C_{\psi_n}&=2\pi\int_{-\infty}^{\infty}|\xi|^{-1}|\widehat{\psi_n}(\xi)|^2 d\xi =
	2\pi\int_{-\infty}^{\infty}\frac{1}{|\xi|}\cdot \frac{n}{|B_{2n}|(2^{2n}-1)\pi}\cdot \xi^{2n-2}\cdot \frac{\pi\xi}{\sinh \pi\xi}d\xi\\
	&= \frac{2\pi n}{|B_{2n}|(2^{2n}-1)}\int_{-\infty}^{\infty}\frac{\xi^{2n-2}}{|\sinh \pi\xi|}=
	\frac{2\pi n}{|B_{2n}|(2^{2n}-1)}\cdot \frac{4\Gamma(2n-1)}{\pi^{2n-1}}\Big( 1-\frac{1}{2^p}\Big)\zeta(2n-1)\\
	&=\frac{ n(2^{2n-1}-1)\Gamma(2n-1)}{|B_{2n}|(2^{2n}-1)\cdot 2^{2n-4}\cdot \pi^{2n-2}}\cdot \zeta(2n-1) <\infty,
\end{align}	
	where $\zeta(z)$ is the Riemann Zeta function.

As usually, we generate from the mother wavelet a doubly-indexed family of wavelets from $\psi_n$ by dilating and translating
	\[\psi_n^{a,b}(t)=\frac{1}{\sqrt{a}}\psi_n\Big(\frac{t-b}{a}\Big),
\]
where $a,b\in \mathbb{R}, \; a > 0, \; n=2,3,\ldots$.

\section{Applications}\label{sec4}
We will look, in a time series  $(y_n)$, for points corresponding to zeros of the second or the third derivative of the Gompertz function. This is equivalent to detecting points, where the sequence of second differences, 
	\[\Delta^2y_n=y_{n+1}-2y_n+y_{n-1},
\]
takes a value close to zero or a maximum respectively.

To calculate the CWT (Continuous Wavelet Transform) coefficients for $(\Delta^2y_n)$, we implement the mother wavelet $\psi_2(t)$, (\ref{3g}) into Matlab's wavelet toolbox. Two parameters of the Gompertz wavelet, $b$ - shift (translation) and $a$ - dilation, can be read from the CWT scalogram by finding a point where the sum (\ref{5a}) (denoted on the scalogram by Index) is locally maximal. It remains to determine the third parameter of the wave, i.e., its saturation level $y_{max}$. Assume that time series $(y_n)$ locally follows the Gompertz function $
\displaystyle y_n\approx y(n)=y_{max}e^{-e^{-\frac{n-b}{a}}}$. By definition  (\ref{3g}) we have
	
	\[ y''(t)=\frac{y_{max}}{2\sqrt{2}a^{3/2}}\psi_2^{a,b}(t).
\]
\begin{lemma*}
The continuous wavelet transform CWT (\ref{3b}) of the function  $y''(t)$, by using Gompertz wavelets $\psi_2^{c,d}$
\[(T^{wav}y'')(c,d)=\langle y'', \psi_2^{c,d} \rangle =\int_{-\infty}^{\infty} y''(t) \psi_2^{c,d}(t) dt,
\]

takes the maximum value when $c=a$ and $d=b$.
\end{lemma*}
\begin{proof}
By the Cauchy-Schwartz inequality
\[|(T^{wav}y'')(c,d)|=|\langle y'', \psi_2^{c,d} \rangle |\le ||y''|| \;||\psi_2^{c,d}||=\frac{y_{max}}{2\sqrt{2}a^{3/2}}||\psi_2^{a,b}||\;||\psi_2^{c,d}||=\frac{y_{max}}{2\sqrt{2}a^{3/2}}.
\]
However the maximum is reached for $c=a$, $d=b$, because:
\[(T^{wav}y'')(a,b)=\langle y'', \psi_2^{a,b} \rangle =\frac{y_{max}}{2\sqrt{2}a^{3/2}}\langle \psi_2^{a,b} , \psi_2^{a,b} \rangle=\frac{y_{max}}{2\sqrt{2}a^{3/2}}.\]
\end{proof}

In view of Lemma, for the maximal value of Index we get successively

\begin{align}
	\textrm{Index}=\sum\limits_{n}&\Delta^2y_n\psi_2^{a,b}(n)\approx \sum\limits_{n}\Delta^2y(n)\psi_2^{a,b}(n)
	\approx \int_{-\infty}^{\infty}y''(t)\psi_2^{a,b}(t)dt
	 =\int_{-\infty}^{\infty} \frac{y_{max}}{2\sqrt{2}a^{3/2}}\psi_2^{a,b}(t)\psi_2^{a,b}(t)dt \label{5a}\\
	&=\frac{y_{max}}{2\sqrt{2}a^{3/2}}
	\int_{-\infty}^{\infty}(\psi_2^{a,b}(t))^2 dt= \frac{y_{max}}{2\sqrt{2}a^{3/2}}. \nonumber
\end{align}

Using (\ref{5a}) we can estimate the saturation level $y_{max}$ as follows
\begin{equation}\label{5b}
y_{max}\approx 2\sqrt{2}a^{3/2}\sum\limits_{n}\Delta^2y_n\psi_2^{a,b}(n)= 2\sqrt{2}a^{3/2}\textrm{Index}.
\end{equation}

\subsection{The case of two exact Gompertz functions}
For illustration, consider function  $y(t)$ composed of two Gompertz functions, Fig.~\ref{f3a}:
\[y(t)=100000e^{-e^{-\frac{t-25}{8}}}+200000e^{-e^{-\frac{t-200}{20}}}, \quad t\in [0, 350].\]
Denote $y_n=y(n), \;\; n=0,1,2,\ldots 350$, calculate the first differences $\Delta^1y_n=y_n-y_{n-1}, \;\; n=1,2,\ldots 350$ (Fig.~\ref{f3b}) and the central second differences $\Delta^2y_n=y_{n+1}-2y_n+y_{n-1}, \;\; n=1,2,\ldots, 349$ (Fig.~\ref{f3c}) .

\begin{figure}[H]
\centerline{}
\centerline{}
\centering
\begin{tabular}{ll}
\begin{subfigure}{7cm}
\centering
\includegraphics[width=7cm,height=5cm]{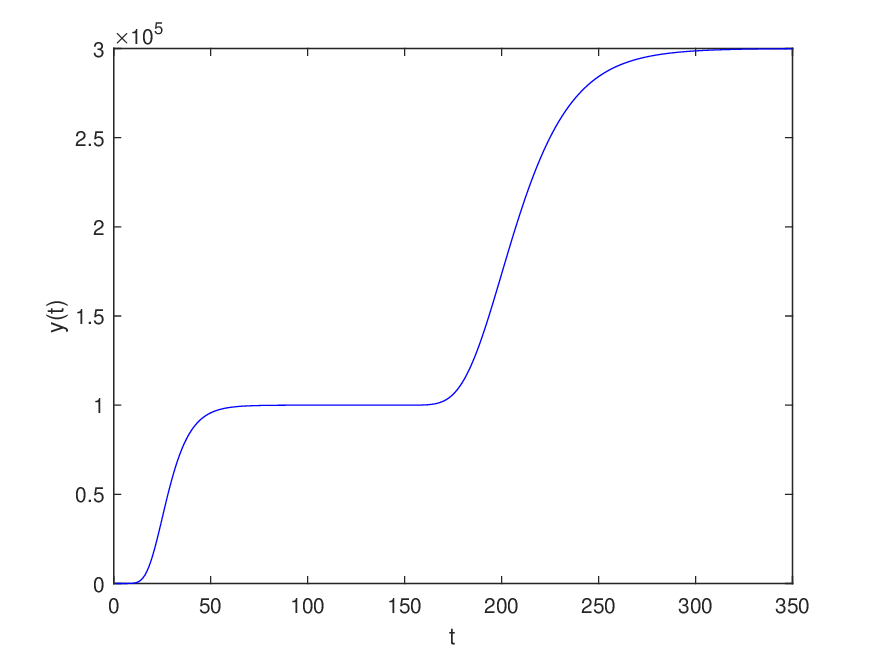}
\caption{Function $y(t)$}
\label{f3a}
\end{subfigure}
&
\begin{subfigure}{7cm}
\centering
\includegraphics[width=7cm,height=5cm]{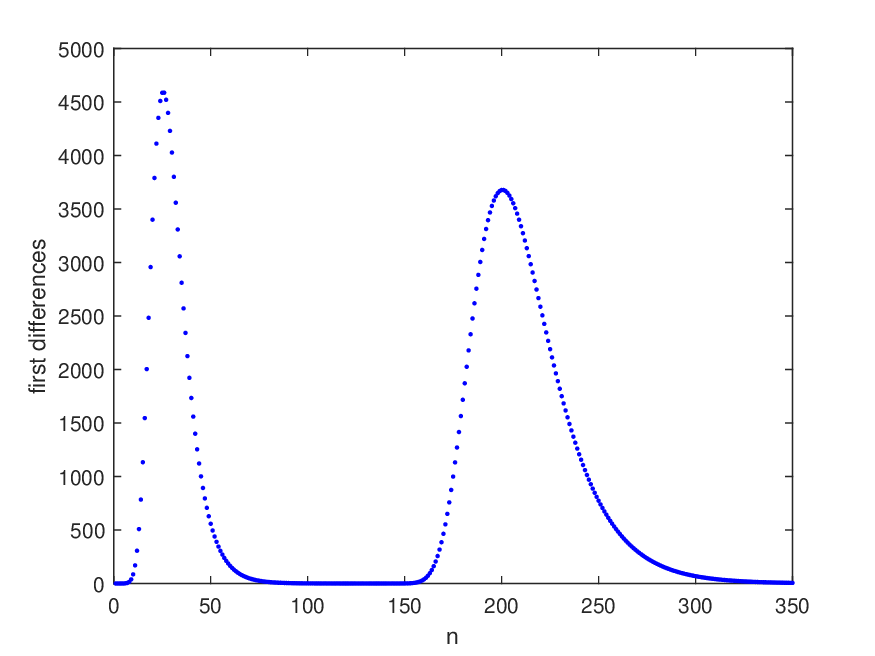}
\caption{First differences $\Delta^1y_n$}
\label{f3b}
\end{subfigure}
\end{tabular}
\centerline{}
\centerline{}
\centering
\begin{tabular}{ll}
\begin{subfigure}{7cm}
\centering
\includegraphics[width=7cm,height=5cm]{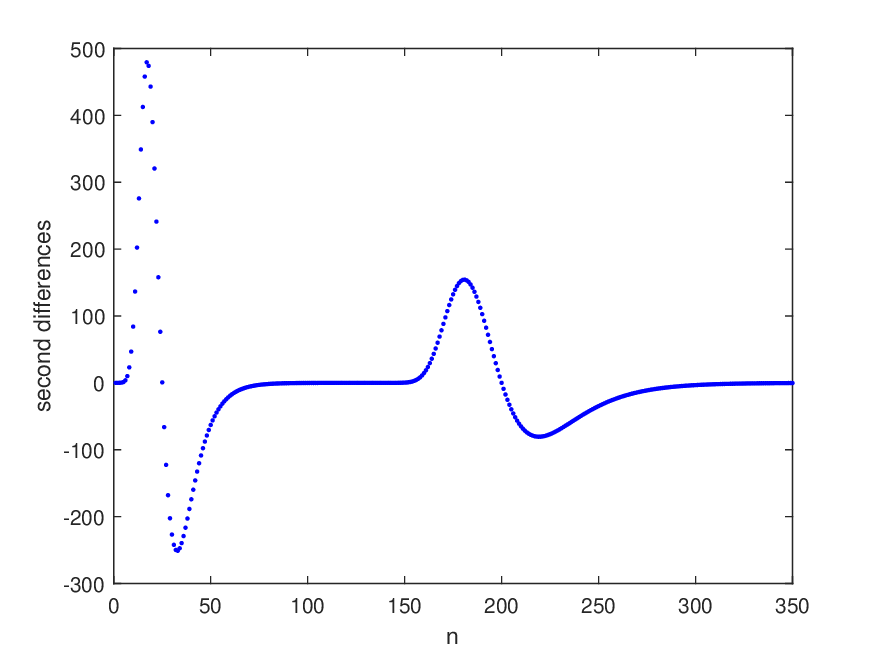}
\caption{ Second differences  $\Delta^2y_n$}
\label{f3c}
\end{subfigure}
&
\begin{subfigure}{7cm}
\centering
\includegraphics[width=7cm,height=5cm]{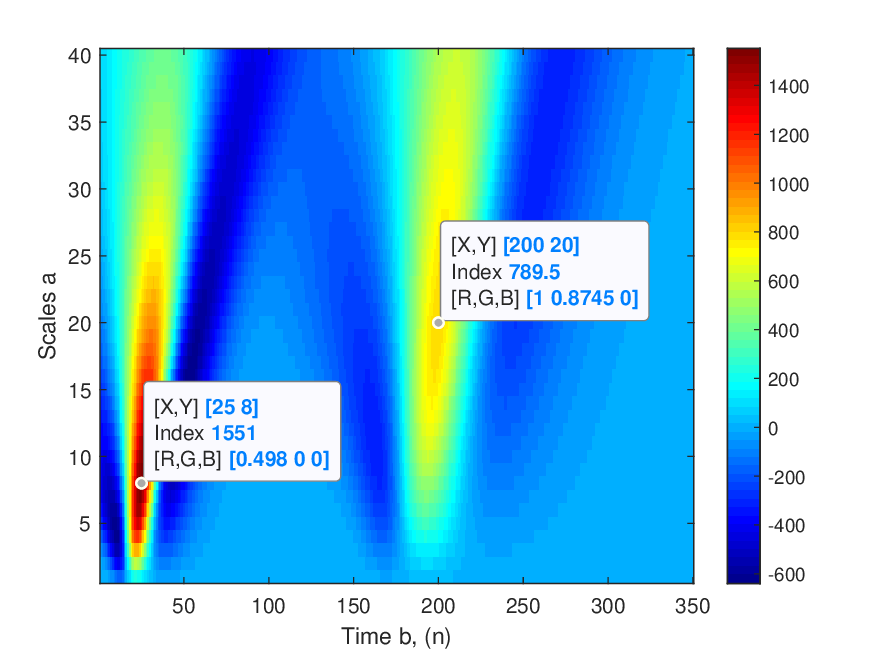}
\caption{CWT scalogram for second differences}
\label{f3d}
\end{subfigure}
\end{tabular}
\caption{Graphs and the CWT scalogram for the sum of two exact Gompertz waves}
\label{fig3}
\end{figure}
The CWT applied to the second differences produces scalogram, shown in  Fig.~\ref{f3d}. We can read from it two parameters $b=25, a=8$ for the first Gompertz wave and similarly $b=200, a=20$ for the second. At these points, the value of the Index is locally the largest. Both the saturation levels can be calculated using formula (\ref{5b}). The saturation level for the first wave is
\begin{equation}\label{5c}
y_{max}\approx 2\sqrt{2}a^{3/2}\sum\limits_{n}\Delta^2y_n\psi_2^{a,b}(n)=2\sqrt{2}\cdot 8^{3/2}\cdot 1551=99,264,
\end{equation}
and analogously for the second
\begin{equation}\label{5d}
y_{max}\approx 2\sqrt{2}a^{3/2}\sum\limits_{n}\Delta^2y_n\psi_2^{a,b}(n)=2\sqrt{2}\cdot 20^{3/2}\cdot 789.5=199,729.
\end{equation}
Since the second differences (not the second derivatives) were used for the analysis, then the calculated saturation levels (28) and (29) cannot be expected to be exactly equal to the theoretical ones, i.e., 100,000 and 200,000 respectively.

\subsection{Application for the analysis of the spread of Covid-19 cases on the example of Saudi Arabia}
Rz\c{a}dkowski and Figlia \cite{RF} presented the use of logistic wavelets to model the spread of the Covid-19 pandemic in several countries. It turns out that sometimes it is better to use the Gompertz curve, than the logistic one, to model particular waves of the pandemic. Dhabi et al. \cite{DCB} modeled, using the Gompertz function, the extensive first wave of Covid-19 cases in Saudi Arabia. They considered 264 days starting from March 12, 2020, until November 30, 2020.

Let us now examine the spread of Covid-19 for Saudi Arabia over a longer time period, from March 12, 2020, to July 20, 2022 covering 861 days. The time series of the total number of reported infections has been smoothed with a 7-day moving average and then denoted by $(y_n)$. Therefore, $n = 1$ in the series $(y_n)$ means March 18, 2020, and $n = 854$ is the last day covered by the analysis, i.e., July 19, 2022 (necessity to calculate the last second difference). Fig.~\ref{fig4} shows in turn: the series $(y_n)$ (Fig.~\ref{f4a}), first differences (Fig.~\ref{f4b}) and second differences (Fig.~\ref{f4c}).

\begin{figure}[H]
\centerline{}
\centerline{}
\centering
\begin{tabular}{ll}
\begin{subfigure}{7cm}
\centering
\includegraphics[width=7cm,height=5cm]{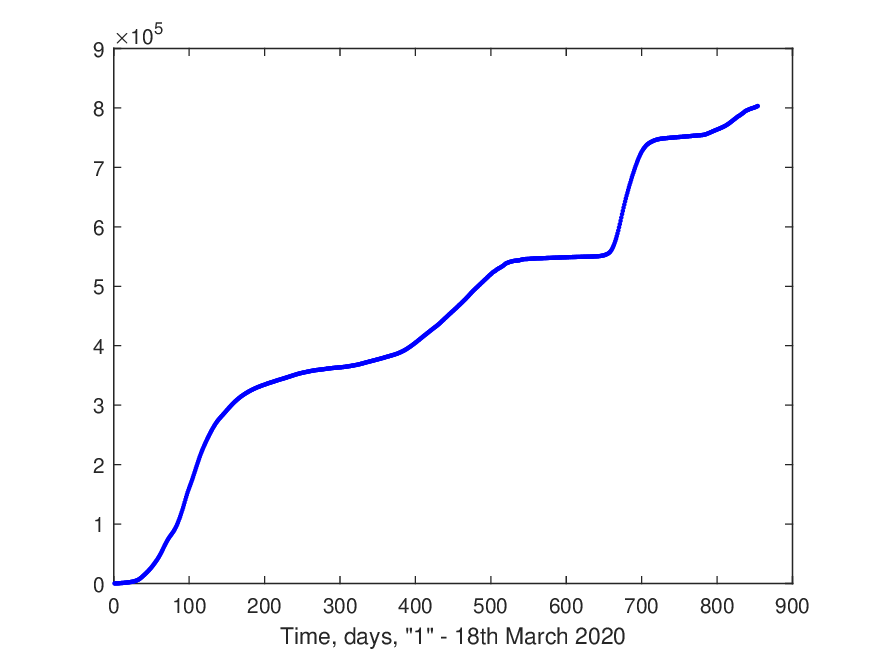}
\caption{Total number of cases in Saudi Arabia, ($y_n$)}
\label{f4a}
\end{subfigure}
&
\begin{subfigure}{7cm}
\centering
\includegraphics[width=7cm,height=5cm]{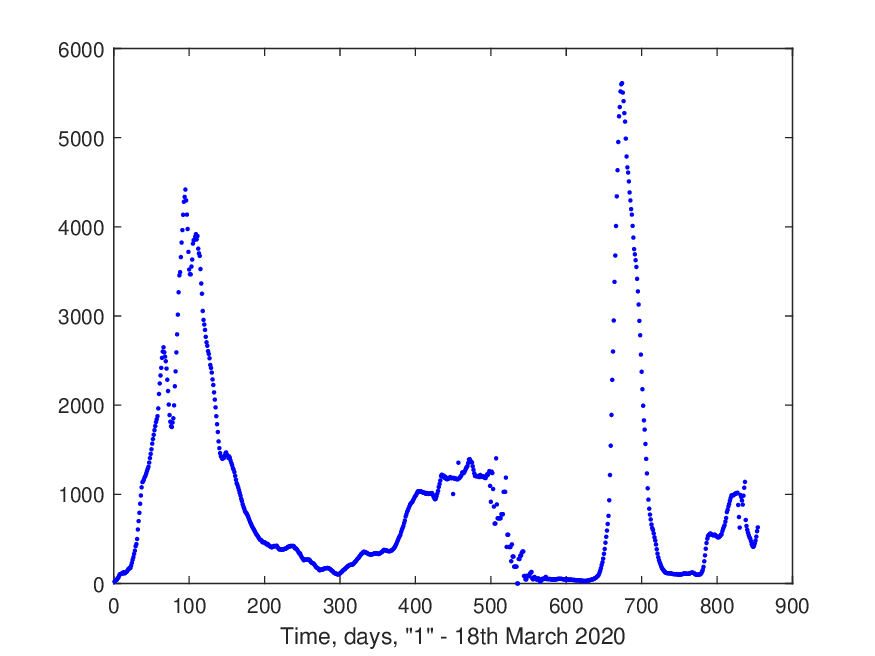}
\caption{First differences, daily cases}
\label{f4b}
\end{subfigure}
\end{tabular}
\centerline{}
\centerline{}
\centering
\begin{tabular}{ll}
\begin{subfigure}{7cm}
\centering
\includegraphics[width=7cm,height=5cm]{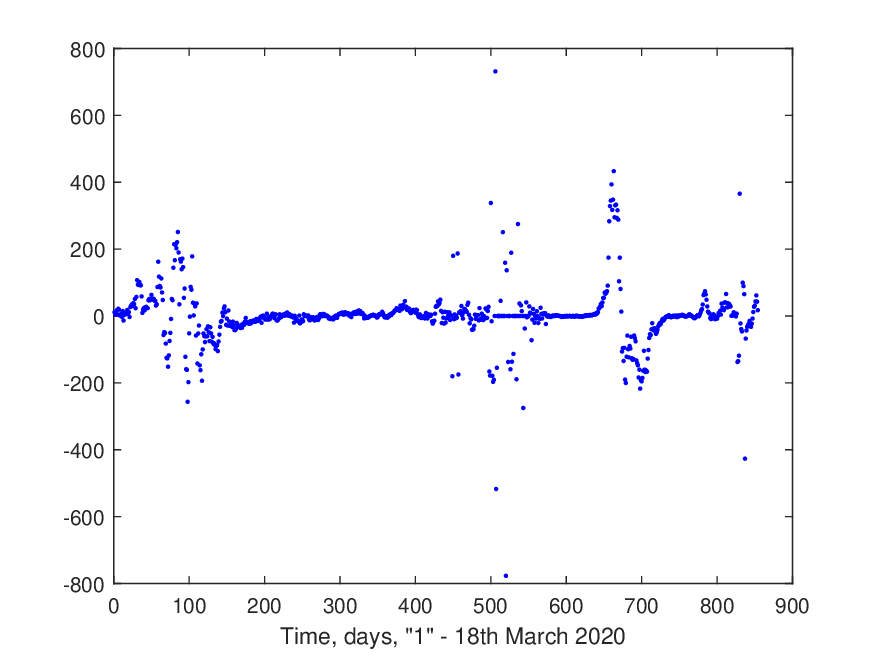}
\caption{ Second differences}
\label{f4c}
\end{subfigure}
&
\begin{subfigure}{7cm}
\centering
\includegraphics[width=7cm,height=5cm]{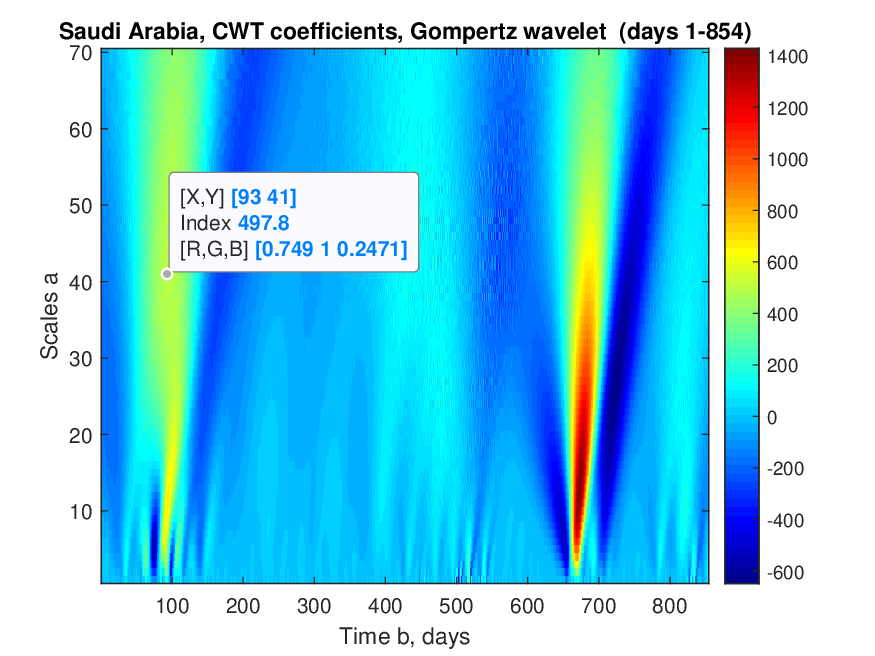}
\caption{CWT scalogram for second differences (c) }
\label{f4d}
\end{subfigure}
\end{tabular}
\caption{Graphs and the CWT  (Gompertz wavelet) scalogram for Saudi Arabia}
\label{fig4}
\end{figure}
The scalogram Fig.~\ref{f4d} gives values of the CWT coefficients for the second differences, using Gompertz wavelets. At the point indicated by the label, there is a maximum of Index of the large wave (consisting of several smaller waves) of cases, studied by Dhahbi et al. \cite{DCB}. The saturation level of this wave calculated by formula (\ref{5b}) is
\begin{equation}\label{6a}
y_{max}\approx 2\sqrt{2}a^{3/2}\sum\limits_{n}\Delta^2y_n\psi_2^{a,b}(n)=2\sqrt{2}\cdot 41^{3/2}\cdot 497.8=369,637.
\end{equation}
The saturation level (\ref{6a}) is consistent with the estimates given by Dhahbi et al.  \cite{DCB} and in accordance with Fig.~\ref{f4a}. We  have to remember that,  this time, the pandemic had been not yet saturated.

Let us examine now the large single wave of cases, which is visible in Fig.~\ref{fig4} after day 600. The asymmetric shapes of the first differences Fig.~\ref{f4b} and the second differences Fig.~\ref{f4c} indicate that the Gompertz function could also be here more efficient for modeling than the logistic curve. This is confirmed by the CWT analysis and can be seen in scalograms Fig.~\ref{fig5}. On the left, there is the CWT analysis using logistic wavelets Fig.~\ref{f5a} and on the right - using Gompertz wavelets Fig.~\ref{f5b}.

\begin{figure}[H]
\centerline{}
\centerline{}
\centering
\begin{tabular}{ll}
\begin{subfigure}{7cm}
\centering
\includegraphics[width=7cm,height=5cm]{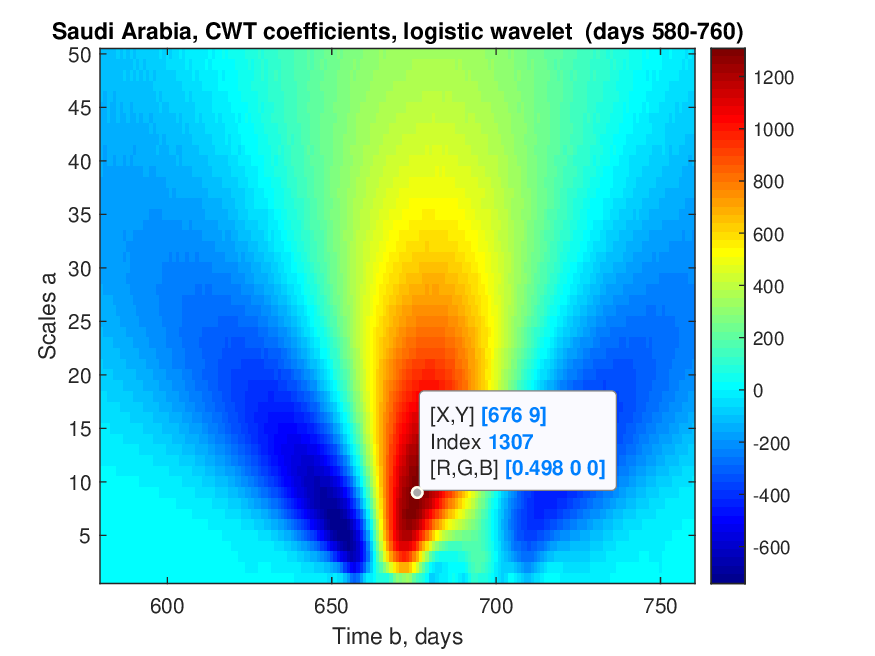}
\caption{CWT analysis, logistic wavelet}
\label{f5a}
\end{subfigure}
&
\begin{subfigure}{7cm}
\centering
\includegraphics[width=7cm,height=5cm]{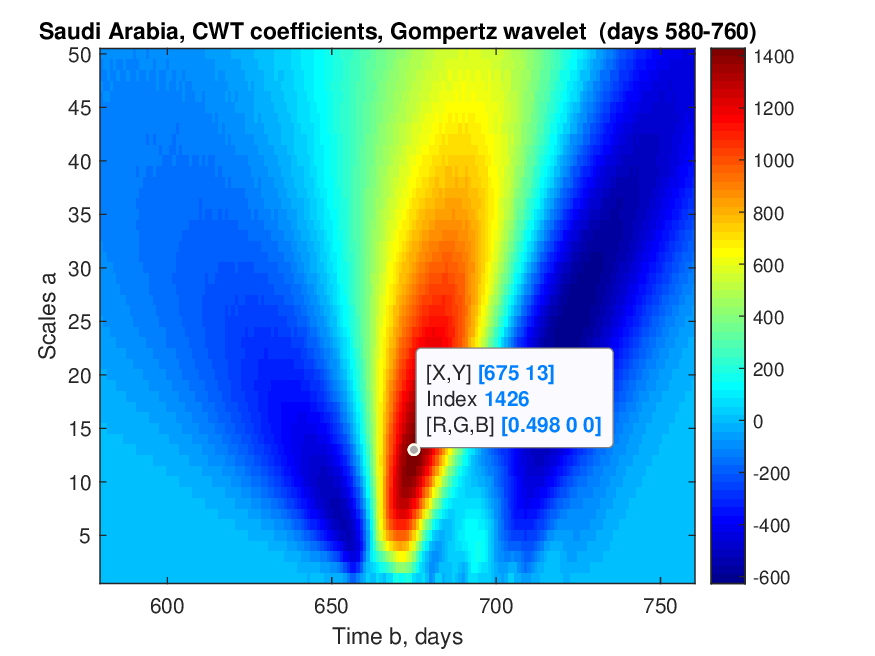}
\caption{CWT analysis, Gompertz wavelet}
\label{f5b}
\end{subfigure}
\end{tabular}
\caption{Comparison of CWT coefficients using the logistic wavelet and the Gompertz wavelet}
\label{fig5}
\end{figure}
 
 The maximum Index  of $1307$ for the logistic wavelet is smaller than the maximum Index of $1426$ for the Gompertz wavelet. Note that both wavelets have the same $L^2$ norm equal to 1. This indicates that the Gompertz wavelet gives a better fit for the points of observation than the logistic one. The saturation level of the wave of cases, calculated by using the formula (\ref{5b}) is

\begin{equation*}
y_{max}\approx 2\sqrt{2}a^{3/2}\sum\limits_{n}\Delta^2y_n\psi_2^{a,b}(n)=2\sqrt{2}\cdot 13^{3/2}\cdot 1426=189,051
\end{equation*}
which is consistent with the observations, as well is the day $n=675$ (January 21, 2022) of the highest smoothed number of cases.

\section{Conclusions and further work}\label{sec5}
In this paper we defined Gompertz wavelets of any order. We have shown that the admissibility condition holds for them. We also calculated their normalizing factors in the space of square intergrable functions  $L^{2}(\mathbb{R})$ and showed that they are expressed by an explicit formula in terms of Bernoulli numbers. Next, we illustrated the utility of second-order Gompertz wavelets in the theoretical situation, where the signal consists of two exact Gompertz functions. Then we used them to study the spread of the Covid-19 pandemic in Saudi Arabia.

In further work, we plan to apply the second-  or  higher-order Gompertz wavelets to some real-world data from the fields such as economics, finance or biology. One could also deal with some generalizations of the basic S-shaped curves to define the corresponding wavelets for them. For this purpose, formulas from the paper Rz\c{a}dkowski and Urlińska  \cite{Rz3} could be used, allowing for efficient computation of successive derivatives for a large class of S-shaped functions.

\vspace{2mm}

{\large\textbf{Conflict of Interests}}  

The author declares that there are no any conflict of interest related to the submitted manuscript.

\vspace{2mm}
{\large\textbf{Funding statement}}

The research was partially funded by the ’IDUB against COVID-19’ project granted by the Warsaw University of Technology (Warsaw, Poland) under the program Excellence Initiative: Research University (IDUB), grant no 1820/54/201/2020.

\end{document}